\documentclass[11pt]{amsart}
\usepackage{etex}
\usepackage{diagbox}
\usepackage{array}
\usepackage{graphicx, pictex}
\usepackage{stmaryrd}
\usepackage{changepage}
\usepackage{float}
\restylefloat{table}
\usepackage{multirow}
\usepackage[dvipsnames]{xcolor}
\usepackage{amsmath,amssymb,amsthm,mathrsfs}
\allowdisplaybreaks
\usepackage{yfonts}

\renewcommand{\d}{\partial}

\def\d{\Omega}

\def\du#1#2#3{\overset{#3}{\underset{#2}{#1}}}
\def\Forall{\quad \hbox{ for all }}

\def\M{{\mathcal{M}}}

\newcommand{\tn}[1]{\lVert\kern-1pt\lvert{#1}\rvert\kern-1pt\rVert}
\def\<{{\langle}}
\def\>{{\rangle}}
\def\Forall{\quad \hbox{ for all }}

\def\d{\Omega}

\def\Forall{\quad \hbox{ for all }}

\def\d{\Omega}

\def\Forall{\quad \hbox{ for all }}

\def\tb#1{{\|\kern-1pt| #1 \|\kern-1pt|}}
\def\nm2#1#2{\|#1\|_{2,\d_{#2}}}

\def\R{\mathbb{R}}

 \theoremstyle{plain}
 \newtheorem{thm}{Theorem}[section]
 \numberwithin{equation}{section} 
 \numberwithin{figure}{section} 
 \theoremstyle{plain}
 \newtheorem{prop}[thm]{Proposition}
 \theoremstyle{plain}
 \theoremstyle{plain}
 \newtheorem{theorem}[thm]{Theorem}
 \theoremstyle{plain}
 
\theoremstyle{plain}
 \newtheorem{remark}[thm]{Remark}
 \theoremstyle{plain}
 
\def\M{{\mathcal{M}}}

\def\d{{\Omega}}

\def\Forall{\quad \hbox{ for all }}

\def\<{{\langle}}
\def\>{{\rangle}}
\def\R{\mathbb{R}}

\def\du#1#2#3{\overset{#3}{\underset{#2}{#1}}}

\usepackage{amssymb}

\begin{document}

\title[Optimal Approximation for Convection-Diffusion]
{Comparison of variational discretizations for a convection-diffusion problem}

\author{Constantin Bacuta}
\address{University of Delaware,
Mathematical Sciences,
501 Ewing Hall, 19716}
\email{bacuta@udel.edu}

\author{Cristina Bacuta}
\address{University of Delaware,
Mathematical Sciences,
501 Ewing Hall, 19716}
\email{crbacuta@udel.edu}

\author{Daniel Hayes}
\address{University of Delaware,
Department of Mathematics,
501 Ewing Hall 19716}
\email{dphayes@udel.edu}

\keywords{finite element, Petrov-Galerkin, upwinding, convection dominated problem, singularly perturbed problems}

\subjclass[2010]{35K57, 35Q35, 65F, 65H10, 65N12, 65N22, 65N30, 74S05, 76R50}
\thanks{The work was supported  by NSF-DMS 2011615}%

\begin{abstract} For a model convection-diffusion problem, we obtain new error estimates for a general upwinding  finite element discretization based on bubble modification of the test space.  The key analysis tool is based on finding  representations of the optimal norms on the trial spaces at the continuous and discrete levels. We analyze and compare the standard linear discretization, the saddle point least square and  upwinding Petrov-Galerkin methods. We conclude that the bubble upwinding Petrov-Galerkin method is the  most performant discretization for the one dimensional model. Our results  for the model convection-diffusion problem can be extended for creating new and efficient discretizations for the  multidimensional cases.



\end{abstract}
\maketitle

\section{Introduction}


We consider  the model of a  singularly perturbed  convection diffusion problem: 
Given data represented by $f\in L^2(\Omega)$, we look for  a solution to the problem 
\begin{equation}\label{eq:2d-model}
  \left\{
\begin{array}{rcl}
     -\varepsilon\Delta u +b\cdot \nabla u & =\ f & \mbox{in} \ \ \ \Omega,\\
      u & =\ 0 & \mbox{on} \ \partial\Omega,\\ 
\end{array} 
\right. 
\end{equation} 
for a  positive constant  $\varepsilon$ and  a bounded domain $\Omega\subset \R^d$. We assume $\varepsilon \ll1$, and $b$ is a given vector chosen such that a unique solution exists. 

For the one dimensional case, we assume that $f$ is a continuous  function on $[0, 1]$, and we look for a solution  $u=u(x)$  such that
\begin{equation}\label{eq:1d-model}
\begin{cases}-\varepsilon u''(x)+ b\, u'(x)=f(x),& 0<x<1\\ u(0)=0, \ u(1)=0,\end{cases}
\end{equation} 
where $b$ is a positive constant. 
Without loss of generality, we  will further  assume that $b=1$. 
The model problems \eqref{eq:2d-model}   and  \eqref{eq:1d-model} arise  when solving   heat transfer problems on thin domains, as well as when using small step sizes in implicit time discretizations of parabolic convection diffusion type problems, see \cite{Lin-Stynes12}. The solutions of these two  problems are characterized by  boundary layers,  see e.g., 
\cite{ EEHJ96,  linssT10, Roos-Schopf15, zienkiewicz2014}. 
Approximating such solutions poses  numerical  challenges due to the $\varepsilon$-dependence of the  stability constants and  of the  error estimates.  
There is a tremendous amount of literature addressing  these types of problems, see e.g.  \cite{EEHJ96, linssT10, quarteroni-sacco-saleri07, Roos-Schopf15, zienkiewicz2014, Dahmen-Welper-Cohen12, CRD-results, CB2}. 
In  this paper we analyze  mixed variational discretizations of the model convection diffusion problem  \eqref{eq:1d-model}, based on the concept of optimal trial  norms  at the continuous and the discrete levels. The concept of optimal trial norm  was developed  and used before in e.g.,  \cite{BHJ21, BHJ22, CRD-results, Broersen-StevensonDPGcd14, Chan-Heuer-Bui-Demkowicz14, Dahmen-Welper-Cohen12, demkowicz-gopalakrishnanDPG10, Barrett-Morton81}. In this paper, for certain discrete test  spaces, we find new representations of such norms that allow  for  sharp error estimates  and new analysis for saddle point or mixed variational formulations.  

We start by  reviewing the standard finite element discretization and two mixed variational formulations  that are known as the Saddle Point Least Square (SPLS) and the Upwinding Petrov-Galerkin (UPG) methods.  We present  new error  analysis results  for the UPG method  and discuss the advantages and disadvantages of the two mixed methods.  The goal of the paper is to develop a set of tools and ideas  for robust discretization of \eqref{eq:1d-model} towards  building efficient new methods for the multidimensional version of convection dominated problems, such as \eqref{eq:2d-model}. 

In Section \ref{sec:ReviewLSPP}, we review  the main concepts and notation for the general standard and  mixed variational formulation  and discretization. 
The general concept of optimal trial space and the main related  results about optimal trial norms is reviewed in  Section \ref{ssec:ON}. We review approximation results for the  standard linear and SPLS discretizations of  \eqref{eq:1d-model} on uniformly distribute notes  in Section \ref{sec:Lin+SPLS}. We justify the {\it oscillatory behvior} of the SPLS method for certain data, in Section \ref{sec:behaveSPLS}. We  present a general approximation result for the UPG method  in  Section \ref{sec:PG}. In Section  \ref{sec:Gen-Bubbles}, we apply the general approximation result of Section \ref{sec:PG} to particular test spaces constructed   with quadratic bubbles and exponential type bubbles. We present a summary of  the ideas and conclusion for the standard and mixed variational formulation of \eqref{eq:1d-model} in Section \ref{sec:conclusion}. 

\section{The general mixed variational approach}\label{sec:ReviewLSPP}
In this section we review  the main concepts and notation for the  mixed variational formulation and discretization. This includes  the  Saddle Point Least Squares (SPLS) method and the particular case of the Petrov-Galerkin (PG) discretization. We follow the  Saddle Point Least Squares (SPLS)  terminology that was introduced in  \cite{BJ-nc, BJprec, BJ-AA,  BJQS17, BQ15, BQ17}. 
\subsection{The abstract variational formulation at the continuous level}  \label{subsec:cont}
We consider the  abstract  mixed  formulation: 
Find $u \in Q$ such that
 \begin{equation}\label{VFabstract}
b(v,u) = \<F, v\>, \ \text{for all} \ v \in V.
\end{equation}
where $b(\cdot,\cdot)$ is a bilinear form, $Q$ and $V$ are possible different  separable Hilbert spaces, and $F$ is a continuous linear functional on $V$. 
We denote the dual of $V$ by $V^*$ and the dual pairing on $V^* \times V$ by $\langle \cdot, \cdot \rangle$.
We assume that the inner products $a_0(\cdot, \cdot)$ and $(\cdot, \cdot)_{{Q}}$ induce the  norms $|\cdot|_V =|\cdot| =a_0(\cdot, \cdot)^{1/2}$ and $\|\cdot\|_{Q}=\|\cdot\|=(\cdot, \cdot)_{Q}^{1/2}$. 
The bilinear form  $b(\cdot, \cdot)$ is a continuous bilinear form on $V\times Q$ satisfying
 the $\sup-\sup$ condition
 \begin{equation} \label{sup-sup_a}
\du{\sup}{u \in Q}{} \ \du {\sup} {v \in V}{} \ \frac {b(v, u)}{|v|\,\|u\|} =M <\infty, 
\end{equation} 
and the $\inf-\sup$ condition
 \begin{equation} \label{inf-sup_a}
 \du{\inf}{u \in Q}{} \ \du {\sup} {v \in V}{} \ \frac {b(v, u)}{|v|\,\|u\|} =m>0.
\end{equation}

 We  assume that the functional  $F \in V^*$ satisfies the {\it compatibility condition} 
\begin{equation}\label{eq:BBsuf}
\<F,v\> =0 \Forall v \in V_0:= \{v \in V |\  b(v, q)=0 \,  \Forall   q \in Q\}. 
\end{equation}
  
The following existence and the uniqueness  result for  \eqref{VFabstract} can be found in e.g.,  \cite{A-B, B09, boffi-brezzi-demkowicz-duran-falk-fortin2006, boffi-brezzi-fortin}. 
\begin{prop} \label{prop:well4mixed} If the form $b(\cdot,\cdot)$ satisfies \eqref{sup-sup_a} and \eqref{inf-sup_a}, and the  data $F \in V^*$ satisfies the {\it compatibility condition}  \eqref{eq:BBsuf}, then  the problem \eqref{VFabstract} has a unique solution that depends continuously on the data $F$. 
\end{prop}
It is also known, see e.g., \cite{BM12, BQ15,BQ17, Dahmen-Welper-Cohen12}, that under the  {\it compatibility condition} \eqref{eq:BBsuf}, solving the mixed problem  \eqref{VFabstract} reduces to  solving a standard saddle point reformulation: Find $(w, u) \in V \times Q$ such that  
\begin{equation}\label{abstract:variational2}
\begin{array}{lclll}
a_0(w,v) & + & b( v, u) &= \langle F,v \rangle &\ \Forall  v \in V,\\
b(w,q) & & & =0   &\  \Forall  q \in Q.  
\end{array}
\end{equation}
In fact,  we have that $u$ is the unique solution of \eqref{VFabstract}  {\it if and only if} $(w=0 , u) $ solves  \eqref{abstract:variational2}.


\subsection{PG and SPLS discretizations}\label{sec:SPLS-discretization}

Let $b(\cdot,\cdot):V\times Q\to\R$ be a bilinear form as defined in Section  \ref{subsec:cont}. 
Let $V_h\subset V$ and  $\M_h\subset Q$ be finite dimensional approximation spaces.
 We assume that the following discrete $\inf-\sup$ condition holds for the pair of spaces $(V_h,\M_h)$:
 \begin{equation} \label{inf-sup_h}
\du{\inf}{u_h \in \M_h}{} \ \du {\sup} {v_h \in V_h}{} \ \frac {b(v_h, u_h)}{|v_h|\,\|u_h\|} =m_h>0.
\end{equation} 
We  define 
$
V_{h,0}:=\{v_h\in V_h\,|\, b(v_h,q_h)=0,\Forall q_h\in \M_h\},
$ 
 and let  $F_h \in V_h^*$ to be the restriction of $F$ to $V_h$, i.e.,   $\langle F_h, v_h \rangle:=\langle F, v_h \rangle$ for all $v_h \in V_h$. 
Consider the following discrete compatibility condition, 
\begin{equation}\label{eq:BBsuf-h}
\langle F,v_h\rangle =0 \Forall v_h\in V_{h,0}.
\end{equation} 
As a direct consequence of proposition \eqref{prop:well4mixed} we have the following result.
\begin{prop} \label{prop:well4mixed-h} If the form $b(\cdot,\cdot)$ satisfies the condition \eqref{inf-sup_h} on $V_h \times \M_h$, and the  data $F_h \in V_h^*$ satisfies the {\it compatibility condition}  \eqref{eq:BBsuf-h}, then  the problem of   finding $u_h\in \M_h$ such that 
\begin{equation}\label{discrete_var_PG}
b(v_h, u_h)=\langle F,v_h\rangle,  \ v_h\in V_h,
\end{equation}
 has a unique solution $u_h \in \M_h$   that depends continuously on the data $F_h$. 
\end{prop}
The variational formulation \eqref{discrete_var_PG} is the {\it Petrov-Galerkin}  (PG) discretization of \eqref{VFabstract}.
We note that for the case  $V_{h,0} =\{0\}$, the compatibility condition \eqref{eq:BBsuf-h} is trivially satisfied. In this case, assuming that $b(\cdot,\cdot)$ satisfies \eqref{inf-sup_h},  the discretization \eqref{discrete_var_PG} leads to a square linear system. Thus, we do not need to consider the SPLS discretization of \eqref{VFabstract}.

 In general, $V_{h,0}$ might not be a subset $V_0$. Consequently,  even though the continuous problem \eqref{VFabstract}  has unique solution, the discrete problem  \eqref{discrete_var_PG} might  not be  well-posed if $F_h$ does not satisfy the {\it compatibility condition}  \eqref{eq:BBsuf-h}. However, if the form $b(\cdot, \cdot)$ satisfies \eqref{inf-sup_h} on $V_h \times \M_h$, then the problem of finding $(w_h,u_h) \in V_h\times \M_h$ satisfying  
\begin{equation}\label{discrete:variationalSPP}
\begin{array}{lclll}
a_0(w_h,v_h) & + & b( v_h, u_h) &= \langle f,v_h \rangle &\ \Forall  v_h \in V_h,\\
b(w_h,q_h) & & & =0   &\  \Forall  q_h \in \M_h, 
\end{array}
\end{equation} 
 does have a unique solution. 
We call  the component $u_h$  of the solution $(w_h,u_h)$ 
of \eqref{discrete:variationalSPP} the  {\it saddle point least squares} approximation of the solution $u$ of the original mixed problem \eqref{VFabstract}.

The following error estimate for $\|u-u_h\|$ was proved in \cite{BQ15}. 

 
\begin{theorem}\label{th:sharpEE} 
Let $b:V \times Q \to \R$  satisfy \eqref{sup-sup_a} and \eqref{inf-sup_a} and assume that   ${F}  \in V^*$  is given and satisfies \eqref{eq:BBsuf}. Assume that  $u$  is the  solution  of \eqref{VFabstract} and  $V_h \subset V$,  $ {\M}_h \subset  Q$ are  chosen such that the discrete $\inf-\sup$ condition   \eqref{inf-sup_h} holds. If  $\left (w_h, u_h \right )$ is the  solution  of \eqref{discrete:variationalSPP}, then the following error estimate holds:
\begin{equation}\label{eq:er4LS} 
\|u-u_h\| \leq  \frac{M}{m_h} \  \du{\inf}{q_h \in\M_h
}{}  \|u-q_h\|.
\end{equation} 
\end{theorem}
 
\begin{remark}\label{th:PGerror}
We note that the estimate \eqref{eq:er4LS} holds true  if $u_h$ is in particular the unique PG solution of \eqref{discrete_var_PG}. This is due to the fact that, if $u_h$ is the solution of \eqref{discrete_var_PG}, then $(0, u_h)$ is the unique solution of \eqref{discrete:variationalSPP}. 
\end{remark} 

For our analysis of the PG discretization of \eqref{eq:2d-model} we will  have a norm 
$ \|\cdot \|_*  $ on $Q$  and a different norm $ \|\cdot \|_{*,h}  $ on the discrete trial space $\M_h$. For this case,  the following  version of  the  Theorem \ref{th:sharpEE}  was proved in \cite{CRD-results}.

\begin{theorem}\label{th:ap-PG}
Let $|\cdot|$, $\|\cdot\|_{*}$ and $\|\cdot\|_{*,h}$ be the norms on $V,Q$, and $\M_h$, respectively, such that  they satisfy \eqref{sup-sup_a}, \eqref{inf-sup_a}, and \eqref{inf-sup_h}. Assume that for some constant $c_0>0$, we have
\begin{equation}\label{eq:c0}
\|v\|_* \leq c_0\|v\|_{*,h}\quad\quad\text{for all $v\in Q$}.
\end{equation}
Let $u$ be the solution of \eqref{VFabstract} and let  $u_h$ be  the unique solution of the problem \eqref{discrete_var_PG}. Then, the following error estimate holds:
\begin{equation}\label{eq:Approx-Opt}
\|u-u_h\|_{*,h}\leq c_0\, \frac{M}{m_h}\ \du{\inf}{p_h \in \M_h}{} \ \|u-p_h\|_{*,h}.
\end{equation}
\end{theorem}


 
 \section{Optimal trial norm for the convection diffusion problem}\label{ssec:ON}
 We consider the variational formulation of \eqref{eq:2d-model}: 
 Find $u\in H_0^1(\Omega)$ such that 
\begin{equation}\label{eq:2d-model-V}
(\varepsilon \nabla u, \nabla v) + (b \cdot \nabla u, v) = (f, v)  \Forall v  \in  H^1_0(\Omega).
\end{equation}

Define $V=Q =H_0^1(\Omega)$  and  $b:V\times Q \to \mathbb{R}$, \ $F\in V^*$ by

\[
b(v,u):=(\varepsilon \nabla u, \nabla v)  + (b \cdot \nabla u, v),  \ \ \ \text{and} \ \ \ \langle F,v\rangle := (f,v).
\]

For the analysis purpose, we will allow different norms on the test and trial spaces.  
 On the test space  $V:=H_0^1(\Omega)$,  we consider the norm induced by 
$ a_0(u,v):=(\nabla u, \nabla v)$. 
We can represent the {\it antisymmetric part} in the {\it symmetric $a(\cdot ,\cdot)$} inner product. 
First, we define the  representation operator $T:Q \to Q$ by 
\[
a_0(Tu, v)  = (b \cdot \nabla u, v),  \Forall v\in V.
\]
In the multidimensional  case we have that
\[
|Tu| = \|b \cdot \nabla u\|_{H^{-1}(\Omega)} \leq \|b\| \|u\|_{L^2(\Omega)}.
\]
 For the one dimensional case and $b=1$, we have 
 \[
 -((Tu)'' , q)= a_0(Tu,q) = (u',q), \ \text{for all} \  q \in Q. 
 \]
By solving the corresponding differential equation, one can find that
\begin{equation}\label{eq:T-Action}
Tu = x\overline{u} - \int_0^xu(s)\,ds, 
\end{equation}
where $\overline{u}=  \int_0^1 u(s)\, ds$. Thus, $(Tu)'(x)= \overline{u} - u(x)$ and 
\begin{equation}\label{eq:T-NormSq}
|Tu|^2 = \int_0^1 |u(s)-\overline{u}|^2\, ds= \|u -\overline{u}\|^2= \|u\|^2 -\overline{u}^2\leq \|u\|^2.
\end{equation}

Next, the  optimal continuous trial norm on $Q$ is defined by 
\[
 \|u\|_{*}:= \du {\sup} {v \in V}{} \ \frac {b(v, u)}{|v|} =   \du {\sup} {v \in V}{} \ \frac {\varepsilon a_0(u, v) + a_0(Tu,v) }{|v|}.
\]
Using the Riesz representation theorem and the fact that $ a_0(Tu,u) =0$, we obtain that the 
optimal trial norm on $Q$  is given by
\begin{equation}\label{eq:COptNormd}
\|u\|_{*}^2  =\varepsilon^2|u|^2 +|Tu|^2.
\end{equation}
 Thus, we have 
  \[
\|u\|^2_*:=\epsilon^2 (\nabla u, \nabla u) + \|b\cdot \nabla u\|_{H^{-1}}^2.
 \]
Using \eqref{eq:T-NormSq}, for the one dimensional case we get 
\begin{equation}\label{eq:COptNorm}
\|u\|_{*}^2 = \varepsilon^2|u|^2 + \|u\|^2 -\overline{u}^2.
\end{equation}


\subsection{Discrete optimal  trial norm} \label{sec:DON}
We assume  that $V_h\subset V=H_0^1(\Omega)$ and $\M_h\subset Q=H_0^1(\Omega)$  are discrete finite element  spaces and that $\M_h\subset V_h$. 
For the purpose of obtaining a discrete optimal norm on $\M_h$, we  let $P_h:Q\to V_h$ be the standard elliptic projection defined by 
\[
a_0(P_h\, u, v_h) = a_0(u,v_h), \ \text{for all} \, v_h \in V_h.
\]
 The optimal trial norm   on $ \M_h$ is 

 \begin{equation}\label{eq:dotn}
  \|u_h\|_{*,h}:= \du {\sup} {v_h \in V_h}{} \ \frac {b(v_h, u_h)}{|v_h|}.
 \end{equation}
Similarly to the continuous case, 
\[
 \|u_h\|_{*,h}:=  \du {\sup} {v_h \in V_h}{} \ \frac {\varepsilon a_0(u_h, v_h) + a_0(Tu_h,v_h) }{|v_h|}  = \du {\sup} {v_h \in V_h}{} \ \frac {\varepsilon a_0(u_h, v_h) + a_0(P_h\,Tu_h,v) }{|v_h|}, 
\]
From the definition of $P_h$ and the anti-symmetry of $T$, we have
\[
\ a_0(P_h\, Tu_h, u_h) = a_0(T u_h,u_h)=0.
\]
Thus, by using the Riesz representation theorem on $V_h$, we get 
\begin{equation}\label{eq:COptNorm-h}
\|u_h\|_{*,h}^2  =\varepsilon^2|u_h|^2 +|P_hTu_h|^2 := \varepsilon^2|u_h|^2 +|u_h|^2_{*,h}.
\end{equation}
Note that for the given  trial  spaces  $\M_h$  and $Q$, the above norm is well defined for any $u\in Q$. Hence, the continuous and discrete optimal trial norms can be compared on $Q$.

The advantage of using  the  optimal trial norm on $Q$ and $\M_h$   resides with the fact that both $inf-sup$ and $sup-sup$  are equal to one at both the continuous and the discrete levels. As a direct consequence of Theorem \ref{th:ap-PG}, we obtain:
 \begin{theorem}\label{th:ap-PG}
Let  $\|\cdot\|_{*}$ and $\|\cdot\|_{*,h}$ be the norms on $Q$, and $\M_h$ and  assume that \eqref{eq:c0} holds. Let $u$ be the solution of \eqref{VF1d} and let  $u_h$ be  the unique solution of problem \eqref{discrete_var_PG}. Then the following error estimate holds:
\begin{equation}\label{eq:Approx-Opt}
\|u-u_h\|_{*,h}\leq c_0\, \du{\inf}{p_h \in \M_h}{} \ \|u-p_h\|_{*,h}.
\end{equation}
\end{theorem}

 \subsection{Discrete optimal  trial norm for the one dimensional case} \label{sec:optimalnorm-CDc}
We review some formulas and results from  \cite{CRD-results, connections4CD}. 

For  $V=Q=H^1_0(0,1)$ we consider the  standard inner product  given by  $a_0(u,v) = (u,v)_V = (u',v')$.
We  divide the interval $[0,1]$ into $n$  equal length subintervals using the nodes $0=x_0<x_1<\cdots < x_n=1$ and denote  $h:=x_j - x_{j-1}, j=1, 2, \cdots, n$. We define  the corresponding finite element discrete space  $\M_h$  as  the spaces of all {\it continuous piecewise linear  functions} with respect to the given nodes, that are zero at $x=0$ and $x=1$. Next, we let  $\M_h=V_h $  be the standard space of continuous piecewise linear functions. 

For the purpose of error analysis, on $V_h$ we consider the standard norm induced by $a_0(\cdot,\cdot)$, but on $\M_h$ we choose an optimal norm  from the stability point of view. 
On $V_h\times \M_h$, we consider the  bilinear from
 \begin{equation}\label{dVFad}
b_d(v_h, u_h)=  d\, a_0(u_h, v_h)+(u'_h,v_h)\  \text{for all} \ u_h\in \M_h, v_h \in V_h,
\end{equation}
where $d=d_{\varepsilon,h}$ is a constant, that might depend on $h$ and $\varepsilon$. 
 The same arguments used in  Section \ref{sec:DON} to deduce the formula \eqref{eq:COptNorm-h}, can be used  here with $\varepsilon=d$ to obtain 
 \begin{equation}\label{eq:COptNorm-h-d}
 \|u_h\|^2_{*,h}=\du {\sup} {v_h \in \M_h}{} \ \frac {(b_d(w_h, u_h))^2}{|w_h|^2}  =d^2|u_h|^2 +|P_hTu_h|^2.
\end{equation}
Denoting $|u|_{*,h}:= |P_hTu|$, see \cite{CRD-results}, we obtain the explicit formula
 \begin{equation}\label{eq:PhTu-Norm} 
|u|^2_{*,h}:= |P_hTu|^2 = \frac{1}{n}\sum_{i=1}^n \left(\frac{1}{h}\, \int_{x_{i-1}}^{x_i} u(x)\,dx\right)^2 - \left(\int_0^1u(x)\,dx\right)^2.
\end{equation}

Using a Poincare inequality,  see \cite{CRD-results} for details, we have
 \begin{equation}\label{eq:CvsD}
\|u\|_{*}^2 -\left (\varepsilon^2 + \frac{h^2}{\pi^2}\right )  |u|^2= \|u-\overline{u}\|^2  -   \frac{h^2}{\pi^2}  \,  |u|^2\leq 
|u|^2_{*,h} \leq \|u\|^2.  
\end{equation}

\section{Standard and and SPLS finite element  variational formulation and discretization}\label{sec:Lin+SPLS}
In this section  we review results for the standard and the SPLS  finite element discretization of \eqref{eq:1d-model}.  In addition, we justify  the oscillatory behavior of the $P^1-P^2$ SPLS discretization. We will use the following  notation:
\[ 
\begin{aligned}
a_0(u, v) & = \int_0^1 u'(x) v'(x) \, dx, \ (f, v) = \int_0^1  f(x) v(x) \, dx,\ \text{and}\\
b(v, u)& =\varepsilon\, a_0(u, v)+(u',v)  \ \text{for all} \ u,v \in V:=H^{1}_0(0,1).
\end{aligned}
 \]
A variational formulation of \eqref{eq:1d-model}, with $b=1$, is as follows: \\ Find $u \in V:= H_0^1(0,1)$ such that
 \begin{equation}\label{VF1d}
b(v,u) = (f, v), \ \text{for all} \ v \in V=H^{1}_0(0,1).
\end{equation}

\subsection{Standard  discretization with $C^0-P^1$ test and trial spaces}\label{sec:1d-lin-discrete}

We  divide the interval $[0,1]$ into $n$  equal length subintervals using the nodes $0=x_0<x_1<\cdots < x_n=1$ and denote  $h:=x_j - x_{j-1}, j=1, 2, \cdots, n$. For the above uniform distributed notes on $[0, 1]$, we define  the corresponding finite element discrete space  $\M_h$  as  the subspace of $H^1_0(0,1)$, given by
 \[ 
 \M_h = \{ v_h \in V \mid v_h \text{ is linear on each } [x_j, x_{j + 1}]\},
 \]
  i.e., $\M_h$ is the space of all {\it continuous piecewise linear  functions} with respect to the given nodes, that are zero at $x=0$ and $x=1$.  We consider the nodal basis $\{ \varphi_j\}_{j = 1}^{n-1}$ with the standard defining property $\varphi_i(x_j ) = \delta_{ij}$.  
We couple the above discrete trial space with the  discrete  test space $V_h:=\M_h$.  
 Thus, the standard $C^0-P^1$  variational formulation of \eqref{VF1d} is: \\ Find $u_h \in \M_h$ such that
 \begin{equation}\label{dVF}
b(v_h, u_h) =\varepsilon (u_h', v_h') + (u_h', w_h) = (f, v_h), \ \text{for all} \ v_h \in V_h.
\end{equation}

 From \eqref{eq:CvsD} it is easy to obtain the following estimate
 \begin{equation}\label{eq:eqNorms2}
\|u\|^2_{*}\leq \left (1+ \left (\frac{h}{\pi \, \varepsilon} \right )^2\right )  \|u\|^2_{*,h}\ \text{for all} \  u\in Q.
\end{equation}

As a consequence of  Theorem \ref{th:ap-PG} and \eqref{eq:eqNorms2}, we have the following result.
\begin{theorem}\label{th:opt-lin}
If $u$ is the solution of \eqref{VF1d}, and $u_h$ the solution of the  linear discretization \eqref{dVF}, then 
\[
\|u-u_h\|_{*,h}  \leq c(h,\varepsilon) \du{\inf}{v_h \in V_h}{} \ \|u-v_h\|_{*,h}, \ \text{where} 
\]
\[
c(h,\varepsilon)= \sqrt {1+ \left (\frac{h}{\pi \, \varepsilon}\right )^2 } \approx \frac{h}{\pi \, \varepsilon} \ \text{ if }\ \varepsilon <<h.
\]
\end{theorem}

In the next sections, we will show that the optimal discrete norm and  $c(h,\varepsilon)$ improve as we consider different test spaces. Numerical  tests for the  case $\int_0^1  f(x)  \, dx\neq 0$, show that as $\varepsilon <<h$, the linear finite element solution of \eqref{dVF}  presents non-physical oscillations, see \cite{CRD-results}.
The behavior of the standard linear finite element approximation  of \eqref{dVF}  motivates  the use of non-standard discretization approaches, such as the  {\it saddle point   least square} or {\it Petrov-Galerkin} methods.%
 
\subsection{SPLS discretization} \label{sec:SPLS}
A  {\it  saddle point least square}  (SPLS) approach   for solving \eqref{VF1d} has  been used before, for example in \cite{BM12,Dahmen-Welper-Cohen12, CRD-results}. \\  For $V=Q= H^1_0(0,1)$, we look for  finding $(w, u) \in V \times Q$ such that 
\begin{equation}\label{SPLS4model2}
\begin{array}{lclll}
a_0(w,v) & + & b( v, u) &= (f,v ) &\ \Forall  v \in V,\\
b(w,q) & & & =0   &\  \Forall  q \in Q,   
\end{array}
\end{equation}
where
\[
b(v, u) =\varepsilon\, a_0(u, v)+(u',v)= \varepsilon\, (u', v')+(u',v). 
\]

Numerical tests for the  discretization of \eqref{SPLS4model2} with various degree polynomial  test and trail spaces were done in  \cite{Dahmen-Welper-Cohen12, dem-fuh-heu-tia19}. Following \cite{CRD-results}, we review the main  error analysis results for  $\M_h= C^0-P^1:= span\{ \varphi_j\}_{j = 1}^{n-1}$, with   the standard linear nodal functions $\varphi_j$,  and $V_h=C^0-P^2$  on given uniformly distributed nodes on $[0, 1]$. To define  a basis for $V_h$, we consider a bubble function for each interval $[x_{i-1}, x_i], i=1,2, \cdots, n$, defined by
 \[
 B_i:= 4\, \varphi_{i-1}\, \varphi_i, \ \ i=1,2, \cdots, n.
 \] 
Then, we have 
\[
V_h:= span \{ \varphi_j \}_{j = 1}^{n-1} + span \{B_{j}\}_{j = 1}^n. 
\]
The SPLS discretization of \eqref{SPLS4model2} is: Find $(w_h, u_h) \in V_h \times \M_h$ such that 
\begin{equation}\label{SPLS4model-h}
\begin{array}{lclll}
a_0(w_h,v_h) & + & b( v_h, u_h) &= (f,v_h ) &\ \Forall  v_h \in V_h,\\
b(w_h,q_h) & & & =0   &\  \Forall  q_h \in \M_h.   
\end{array}
\end{equation}

 In this case, note that the projection $P_h$  defined in Section \ref{sec:DON}, is the projection on the space $V_h=C^0-P^2$. For any piecewise linear function $u_h \in \M_h$, we have that 
\[
Tu_h = x\overline{u}_h - \int_0^xu_h(s)\,ds
\]
is a continuous piecewise quadratic function. Consequently, $Tu_h \in V_h$, and $P_h\, Tu_h =T u_h$. The optimal discrete norm on $\M_h$ becomes 
\[
\|u_h\|_{*,h}^2 =\varepsilon^2 |u_h|^2 + |Tu_h|^2 =\|u_h\|_{*}^2.
\]
Using the  optimal norm on $\M_h$,  a discrete $\inf-\sup$ condition satisfied, and the problem \eqref{SPLS4model-h} has  a unique solution. 
In addition, for this $P^1-P^2$ SPLS discretization, we can consider the same norm given by 
\[
\|u\|_{*}^2 =\varepsilon^2|u|^2 +\|u-\overline{u}\|^2=\varepsilon^2|u|^2 +\|u\|^2-\overline{u}^2=\|u\|_{*,h}^2 
\]
 on both spaces $Q$ and $\M_h$. 
As a consequence of the approximation Theorem \ref{th:sharpEE}, we get the following  optimal error estimate. 

\begin{theorem}\label{th:P1P2err}
If $u$ is the solution of \eqref{VF1d}, and $u_h$ is the  SPLS solution  for the $(P^1-P^2)$ discretization, then  
$$
\|u - u_h\|_{*}\leq \inf_{p_h\in \M_h}\|u - p_h\|_{*} \leq \|u - u_I\|_{*},
$$
where $u_I$ is the interpolant of the exact solution on the uniformly distributed nodes on $[0, 1]$. 
\end{theorem}
\subsection{The oscillatory behavior of the  $P^1-P^2$ SPLS discretization} \label{sec:behaveSPLS}  
For  $\int_0^1f(x)\, dx=0$, the $P^1-P^2$ SPLS discretization improves on the standard linear discretization of \eqref{VF1d} from both the error point of view, and from the presence of the non-physical oscillations point of view. A more detailed  numerical analysis and comparison was done in \cite{CRD-results}.  However, as noted in \cite{CRD-results}, for  $\int_0^1  f(x) \, dx\neq 0$, the SPLS solution $u_h$ approximates the  shift by a constant of  the solution $u$ of \eqref{VF1d} and in addition, {\it non-physical oscillations} still appear in the plot of $u_h$ at the ends of the interval. An explanation of this phenomena can be done using the simplified variational problems. More precisely, we  can consider  the {\it continuous simplified} problem obtained from \eqref{VF1d}, by letting $\varepsilon \to 0$, i.e.,  Find $ u \in Q=H^1_0(0,1)$ such that 
\begin{equation}\label{Smodel-C}
 (u', v) = (f,v ) \ \Forall  v \in V= H^1_0(0,1).
\end{equation}
The problem  is not well posed when $\int_0^1  f(x)  \, dx\neq 0$. In order to have  the existence and the uniqueness of the solution of \eqref{Smodel-C}, we can change the trial space $Q$ to $L^2_0(0, 1):=\{u \in L^2(0, 1) | \ \int_0^1 u=0\}$.  Nevertheless,  in this case, the solution space cannot see the boundary conditions of the original problem \eqref{eq:1d-model}.  
On the other hand, the {\it discrete simplified} linear system obtained from \eqref{SPLS4model-h}  by letting $\varepsilon \to 0$, i.e.: Find $(w_h, u_h) \in V_h \times \M_h$ such that 
\begin{equation}\label{SPLS4model-h-R}
\begin{array}{lclll}
(w'_h,v'_h) & + & (u'_h, v_h) &= (f,v_h ) &\ \Forall  v_h \in V_h,\\
(w'_h,q_h) & & & =0   &\  \Forall  q_h \in \M_h,   
\end{array}
\end{equation}
 has unique solution because a discrete $\inf-\sup$ condition holds when using the optimal trial norm on $\M_h$.  Numerical tests in \cite{CRD-results} showed that oscillation in {the discrete simplified solution}  $u_h$ of \eqref{SPLS4model-h-R} predict oscillatory behavior of  $u_h$-the SPLS discrete solution of \eqref{SPLS4model-h}. In fact for $\varepsilon <<h$, in the ``eye ball measure", the two solutions are identical.  Next, we will justify why the component $u_h$ of \eqref{SPLS4model-h-R} oscillates in the case  $\int_0^1  f(x) \, dx\neq 0$. 
 
 Let $u$  be the  solution of \eqref{Smodel-C} with $Q=L^2_0(0, 1)$ and $V= H^1_0(0,1)$) and let $u_h$  be  the the second component of  the solution of \eqref{SPLS4model-h-R}.  It is easy to check   $u(x)=w(x)-\overline{w}$, where $w(x) =\int_0^x f(s)\, ds$. By eliminating $w_h$ from  the system \eqref{SPLS4model-h-R}, it follows that  $u_h- \overline{u}_h$ is the $L^2$ projection of  $u$  onto $\overline{\M}_h:= \{w_h -\overline{w}_h \ | w_h \in \M_h\}$. We note that $\overline{\M}_h$ is a space of continuous piecewise linear functions that have the same values at the end points of $[0, 1]$, while  $u$ cannot have the same  values at the end points if  $\int_0^1  f(x) \, dx\neq 0$. This explains the  {\it non-physical oscillations}  of the SPLS discretization of \eqref{SPLS4model-h}.

 In the next section, we present a particular SPLS discretization that is free of {\it non-physical oscillations}.

 
\section{The Petrov-Galerkin method  with bubble type test space } \label{sec:PG}
For improving the stability and approximability of the standard linear finite element approximation for solving \eqref{VF1d}, various  Petrov-Galerkin discretizations  were considered, see e.g., \cite{CRD-results, zienkiewicz76, mitchell-griffiths, roos-stynes-tobiska-96, zienkiewicz2014}. In this section, we analyze a general class of Upwinding Petrov-Galerkin   (UPG) discretizations based on a bubble modification of the  standard $C^0-P^1$ test space.  
The idea is to define  $V_h$  by adding to each $\varphi_j$, a pair of polynomial bubble functions.
According to Section 2.2.2 in \cite{zienkiewicz2014}, this idea  was first suggested in \cite{ZGHupwind76} and used in the same year in \cite{zienkiewicz76} with  quadratic bubble modification.  
The method  is known in literature as {\it upwinding PG method} or {\it upwinding finite element method}, see \cite{roos-stynes-tobiska-96, zienkiewicz2014}. Next, we build on the description of UPG introduced in  \cite{connections4CD}   emphasizing on a new error analysis of the method.


The standard  variational formulation for solving  \eqref{eq:1d-model} with $b=1$, is: Find $u \in Q=H^{1}_0(0,1)$ such that 
\begin{equation}\label{PG4model}
b(v, u) =\varepsilon\, a_0(u, v)+(u',v) = (f,v ) \ \Forall  v \in V=H^{1}_0(0,1).\\
\end{equation}
A general  Petrov-Galerkin method for solving \eqref{PG4model} chooses  a test space $V_h \subset V=H^{1}_0(0,1)$ that is different from the trial space $\M_h \subset Q=H^1_0(0,1)$.


For describing the general UPG discretization we  consider a continuous  (bubble)  function $B:[0,h] \to\R$ with the following properties:
\begin{equation}\label{Bbounds}
B(0)=B(h)=0,\\
\end{equation}

\begin{equation}\label{b1}
\int_0^h B(x)  \, dx=b_1h, \ \text{with} \ b_1>0.
\end{equation}

\begin{equation}\label{b2}
\int_0^h (B'(x))^2  \, dx=\frac{b_2}{h}, \ \text{with} \ b_2>0.
\end{equation}
 By translating $B$, we generate $n$ bubble functions that are locally supported. For $ i=1,2, \cdots, n$,  we define $B_i:[0, 1] \to \R$ by $B_i(x)=B(x-x_{i-1})=B(x-(i-1)h)$ on $[x_{i-1}, x_i]$, and we extend it  by zero to the entire interval $[0, 1]$. Note that $B_1=B$ on $[0, h]$, and for $ i=1,2, \cdots, n$, we have
\begin{equation}\label{B_i_bounds}
B_i(x_{i-1})=B_i(x_i)=0, \  \text{and} \ B_i=0\ \text{on} \ [0, 1]\backslash (x_{i-1}, x_i).
\end{equation}

\begin{equation}\label{B_i_b1}
\int_{x_{i-1}}^{x_i} B_i(x)  \, dx=b_1h, \ \text{with} \ b_1>0,\\
\end{equation}
and
\begin{equation}\label{B_i_b2}
\int_{x_{i-1}}^{x_i} (B_i ' (x) )^2 \, dx=\frac{b_2}{h}.\\
\end{equation}
Next, we consider a particular class of Petrov-Galerkin discretizations of the model problem \eqref{PG4model} with trial space $\M_h= span\{ \varphi_j\}_{j = 1}^{n-1}$ and the test space  $V_h$ obtained by modifying $M_h$  using the bubble functions $B_i$. 
We define the test space $ V_h$ by 
 \[
 V_h:= span \{ \varphi_j  + (B_{j}-B_{j+1})\}_{j = 1}^{n-1}, 
 \]
 where $\{ B_i\}_{i=1,\cdots,n}$ satisfy  \eqref{B_i_bounds}-\eqref{B_i_b2}. 
We note that both $\M_h$ and $V_h$ have the same dimension of $(n-1)$. 


The upwinding Petrov Galerkin discretization with general bubble functions for 
\eqref{eq:1d-model}  with $b=1$ is: Find $u_h \in \M_h$ such that 
\begin{equation}\label{eq:1d-modelPG}
b(v_h, u_h) = \varepsilon\, a_0(u_h, v_h)+(u'_h,v_h) =(f,v_h) \ \Forall  v_h \in V_h. 
\end{equation}
As presented in \cite{connections4CD}, we show that the variational formulation \eqref{eq:1d-modelPG} admits a reformulation that uses  a new bilinear form defined on {\it standard  linear finite element spaces}. We let 
\[
u_h= \sum_{j=1}^{n-1} \alpha_j \varphi_j,
\]
and consider a generic test function $v_h$ defined by 
\[
v_h= \sum_{i=1}^{n-1} \beta_i \varphi_i + \sum_{i=1}^{n-1}  \beta_i (B_i - B_{i+1}) = \sum_{i=1}^{n-1} \beta_i \varphi_i + \sum_{i=1}^{n}  (\beta_i - \beta_{i-1}) B_{i},
\]
where, we define $\beta_0=\beta_n=0$. Next we will use the splitting of $v_h$ in a linear part plus a bubble part:
\[
v_h=w_h + B_h, \ \text{with} \ w_h:=  \sum_{i=1}^{n-1} \beta_i \varphi_i  \ \text{and } \  B_h:=\sum_{i=1}^{n}  (\beta_i - \beta_{i-1}) B_{i}.
\]
Based on  formulas \eqref{B_i_bounds}, \eqref{B_i_b1} and \eqref{B_i_b2}, the fact that $u'_h, w'_h $  are constant  on each of  the intervals $[x_{i-1}, x_i]$,  and that $w'_h= \frac{\beta_i -\beta_{i-1} }{h}$ on $[x_{i-1}, x_i]$,  we obtain
\[
(u'_h, B_h) = \sum_{i=1}^{n}\int_{x_{i-1}}^{x_i}  u'_h (\beta_i - \beta_{i-1}) B_{i}=
 \sum_{i=1}^{n} u'_h \,  w'_h \int_{x_{i-1}}^{x_i}  B_{i} =b_1h \sum_{i=1}^{n} \int_{x_{i-1}}^{x_i}  u'_h  w'_h. 
\]
Thus
\begin{equation} \label{eq:upBh}
(u'_h, B_h) =b_1h (u'_h, w'_h), \ \text{where} \  v_h=w_h + B_h.
\end{equation}

In addition, since $u'_h$ is constant on $[x_{i-1},x_i]$, we have
\[
(u'_h, B'_i) =u'_h\int_{x_{i-1}}^{x_i} B'_i(x)  \, dx=\ 0 \ \text{for all} \ i=1, 2, \cdots,  n, \text{hence} 
\]
\begin{equation} \label{eq:upBph}
 (u'_h, B'_h) =0, \  \text{for all} \  u_h \in \M_h, v_h=w_h + B_h \in V_h.
\end{equation}
From  \eqref{eq:upBh} and  \eqref{eq:upBph}, for any $u_h \in \M_h, v_h=w_h + B_h \in V_h$ we get
\begin{equation} \label{eq:bPG}
 b(v_h, u_h) = \left (\varepsilon + b_1h\right )  (u'_h, w'_h) +  (u'_h, w_h).
\end{equation}
Introducing the notation $d=d_{\varepsilon,h} = \varepsilon +h\, b_1$ and using the notation of Section \ref{sec:optimalnorm-CDc},  we have 
\begin{equation} \label{eq:bPGd}
 b(v_h, u_h) = b_d(u_h,w_h), \ \text{where} \ v_h=w_h + B_h, \ \text{and} \ u_h, w_h \in \M_h.
\end{equation}

Using \eqref{eq:upBph} and \eqref{B_i_b2},  we  note that for any $v_h=w_h + B_h \in V_h$  the energy 
norm of $v_h$ is a multiple of the energy of the linear part $w_h$. Indeed, 
 \[
 \begin{aligned}
 (v'_h, v'_h) & =(w'_h + B'_h, w'_h + B'_h) = (w'_h, w'_h) + ( B'_h, B'_h) =\\ 
 & =(w'_h, w'_h) +  \sum_{i=1}^{n} (\beta_i-\beta_{i-1})^2 (B'_i, B'_i)=\\
 & =(w'_h, w'_h)   + b_2h \sum_{i=1}^{n} \left (\frac{\beta_i-\beta_{i-1}}{h} \right )^2 = \\
 & =(w'_h, w'_h)   +  b_2 \sum_{i=1}^{n} \left (\int_{x_{i-1}}^{x_i} (w'_h)^2 \right )^2 =(w'_h, w'_h)   +  b_2 (w'_h, w'_h). 
 \end{aligned}
 \]
 
Consequently,
\begin{equation} \label{eq:vhwh}
|v_h|^2 = (1+b_2)|w_h|^2. 
\end{equation}

The above  remarks, see \eqref{eq:bPGd} and  \eqref{eq:vhwh}, lead to the following result.
 
\begin{theorem}
For the bilinear form $b(\cdot, \cdot)$ of  \eqref{PG4model}  on $\M_h \times V_h$  with the bubble enriched test space $V_h$, the discrete optimal norm  on $\M_h$ is given by  
 \begin{equation}\label{eq:COptNorm-hB} 
\|u_h\|_{*,h}^2  = \frac{(\varepsilon +h\, b_1)^2}{1+b_2}\, |u_h|^2 +  \frac{1}{1+b_2}\ |u_h|^2_{*,h}.
\end{equation}
where  $|u_h|^2_{*,h}$  is defined in \eqref{eq:PhTu-Norm}. 
\end{theorem}

\begin{proof}
Using the definition of $\|u_h\|_{*,h}$  along with the work of  Section \ref{ssec:ON}, we can reduce the supremum over $V_h$ to a supremum over $\M_h$. Indeed, using  the splitting $v_h=w_h + B_h$, \eqref{eq:bPGd}  and \eqref{eq:vhwh} we have 
\[
\|u_h\|^2_{*,h}=\sup_{v_h\in V_h}\frac{(b(v_h,u_h))^2}{|v_h|^2} = \sup_{v_h\in V_h}\frac{(b_d(w_h,u_h))^2}{|v_h|^2}= \sup_{w_h\in \M_h}\frac{(b_d(w_h,u_h))^2}{(1+b_2)|w_h|^2}.
\]
Next, based on  \eqref{eq:COptNorm-h-d}  we arrive at \eqref{eq:COptNorm-hB}.
\end{proof}

\begin{prop}\label{propPG} Assume that $h$ is chosen such that 
\begin{equation}\label{eq:h-restriction}
\varepsilon^2+ \frac{h^2}{\pi^2} \leq (\varepsilon +h\, b_1)^2.
\end{equation}
Then, the following inequality between  $\|u\|_{*}$ and $\|u\|_{*,h}$ holds on  $Q$.
\begin{equation}\label{eq:PG8Ineq}
\|u\|^2_{*} \leq (1+b_2)\,   \|u\|^2_{*,h}, \ \text{for all} \  u \in Q=H^1_0(0, 1).
\end{equation}
\end{prop}
\begin{proof}
Using the inequality  \eqref{eq:CvsD} for $\|u\|_{*}$  and the formula   \eqref{eq:COptNorm-hB} for $\|u\|_{*,h}$, we have 
\[
\|u\|^2_{*} -  (1+b_2)\,   \|u\|^2_{*,h} \leq  \left (\varepsilon^2+ \frac{h^2}{\pi^2} -(\varepsilon +h\, b_1)^2\right ) |u|^2.
\]
Now, under the assumption \eqref{eq:h-restriction}, we obtain \eqref{eq:PG8Ineq}.
\end{proof}

As a consequence, we have the following error estimate.

\begin{theorem}\label{thm:PGError}
If $u$ is the solution of \eqref{VF1d}, $u_h$ the solution of the UPG formulation \eqref{eq:1d-modelPG},  and  $h$ is chosen such that \eqref{eq:h-restriction} holds, then
 
\begin{equation}\label{eq:PGerrB}
\|u - u_h\|_{*,h}\leq \sqrt{1+b_2}\, \inf_{p_h\in \M_h}\|u - p_h\|_{*,h}.
\end{equation}
\end{theorem}
\begin{proof}
The estimate is a direct consequence of the approximation Theorem \ref{th:ap-PG} and the Proposition \ref{propPG}.
\end{proof}

\begin{remark}\label{rem:Lin-sys}
 Based on \eqref{eq:bPG}, the  linear system associated with the UPG method \eqref{eq:1d-modelPG} is 
\begin{equation}\label{1d-PG-ls}
\left ( \left (\frac{\varepsilon}{h} + b_1 \right ) S+ C \right )\, U = F_{pg}, 
\end{equation}
where \(U,F_{pg}\in\R^{n-1}\)  and 
 \[
 U:=\begin{bmatrix}u_1\\u_2\\\vdots\\u_{n-1}\end{bmatrix},\quad F_{PG}:= \begin{bmatrix}(f,\varphi_1)\\ (f,\varphi_2)\\\vdots \\ (f,\varphi_{n-1})\end{bmatrix} + 
 \begin{bmatrix} (f, B_1 -B_2) 
 \\  (f, B_2-B_3)\\ \vdots \\  (f, B_{n-1} -B_n) \end{bmatrix}, 
\]
 \[
S=tridiag(-1, 2, -1), \  \text{and}  \ C= tridiag\left (-\frac12, 0, \frac12 \right ). 
 \]
By using  the notation $
 d= d_{\varepsilon,h} = \varepsilon +h\, b_1$,
the matrix of the finite element  system  \eqref{1d-PG-ls} is 
\begin{equation} \label{eq:M4CDfe} 
M_{fe}= tridiag\left ( -\frac{d}{h} - \frac{1}{2},\  2\, \frac{d}{h},\  -\frac{d}{h} + \frac{1}{2} \right ).
\end{equation}
\end{remark}

\section{Upwinding PG with particular bubble functions} \label{sec:Gen-Bubbles} 


\subsection{Upwinding PG with quadratic bubble functions} \label{sec:Quad-Bubbles} 
We consider the model problem \eqref{VF1d}  with the discrete space 
 $\M_h= span\{ \varphi_j\}_{j = 1}^{n-1}$ and $V_h$ a modification of $\M_h$  using {\it quadratic bubble functions}.
 The method can be found in e.g., \cite{mitchell-griffiths}. In \cite{connections4CD},  we related the {\it quadratic bubble UPG} method to the general upwinding Finite Difference (FD) method and presented ways to improve the performance of upwinding FD methods. In this section, we establish  error estimates  for the  quadratic bubble UPG method.
 
   First, for a parameter $\beta>0$,  we define the bubble function $B$ on $[0, h]$ by
 \[
 B(x)= \frac{4\, \beta}{h^2} x(h-x).
 \]

 Using the function $B$ and the general construction of Section \ref{sec:PG}, we define the set of bubble functions $\{B_1, B_2, \cdots,B_n\}$  on $[0, 1]$ and 
 \[
 V_h:= span \{ \varphi_j + (B_{j}-B_{j+1}) \}_{j= 1}^{n-1}.
 \]
 Elementary calculations show that \eqref{b1} 
 holds with $b_1=\frac{2\, \beta}{3}$,  and  \eqref{b2} holds with $b_2=\frac{16 \beta^2} {3}$. 
  In this case, we have 
  \[
  d=d_{\varepsilon,h} =\varepsilon + h\, b_1=\varepsilon + \frac{2\beta}{3} h, \ \text{and}\ 1+b_2= \frac{19}{3}\beta^2.
  \] 
 According to   \eqref{eq:COptNorm-hB} the optimal norm  on $\M_h$ is given by 
  \begin{equation}\label{eq:COptNorm-hBq} 
\|u_h\|_{*,h}^2  = \frac{3}{19\beta^2}\left ( \left (\varepsilon + \frac{2\beta}{3}\,  h\right )^2\, |u_h|^2 +  |u_h|^2_{*,h} \right ).
\end{equation}

 In this case, we note that, the restriction \eqref{eq:h-restriction} is satisfied for any $h>0$  if for example 
 $ \beta \geq \frac{\sqrt{3}}{2\, \pi} \approx 0.28$. 
  As a consequence,  we have the following result. 

\begin{theorem}\label{thm:PGError-qB}
If $u$ is the solution of \eqref{VF1d}, $u_h$ the solution of the upwinding PG formulation \eqref{eq:1d-modelPG},  with quadratic bubble  test space and  $\beta \geq \frac{\sqrt{3}}{2\, \pi}$, then by using the discrete norm \eqref{eq:COptNorm-hBq}, we have   
\begin{equation}\label{eq:PGerrBq}
\|u - u_h\|_{*,h}\leq \sqrt{\frac{ 19}{3}} \, \beta \inf_{p_h\in \M_h}\|u - p_h\|_{*,h}. 
\end{equation}
Equivalently, by rescaling the estimate, we have
\begin{equation}\label{eq:PGerrBq2}
\begin{aligned}
& \left ( \left (\varepsilon + \frac{2\beta}{3}\, h\right )^2\, |u-u_h|^2 +  |u-u_h|^2_{*,h} \right )  \leq  \\ \leq  & \frac{ 19}{3} \beta^2
\inf_{p_h\in \M_h}   \left ( \left (\varepsilon + \frac{2\beta}{3}\,  h\right )^2\, |u-p_h|^2 +  |u-p_h|^2_{*,h} \right ).
\end{aligned}
\end{equation}
\end{theorem}

For implementation purposes, according to \eqref{eq:M4CDfe}, the matrix of the finite element  system  \eqref{1d-PG-ls} with quadratic bubble upwinding is 
\begin{equation}\label{eq:Mq}
M^q_{fe} =   tridiag\left ( -\frac{\varepsilon}{h} - \frac{2\beta}{3} - \frac{1}{2},\  \frac{2\, \varepsilon}{h}+ \frac{4\beta}{3} ,\   -\frac{\varepsilon}{h} - \frac{2\beta}{3} + \frac{1}{2} \right ). 
\end{equation}
 We note that for  $\beta=0$, we obtain the matrix corresponding to the standard  finite element discretization \eqref{dVF}.
 The case $\beta=1$ was studied in \cite{CRD-results}. The possibility of choosing $\beta=\beta(\varepsilon,h)$  allows for further simplification.

 \subsection{Special cases for quadratic bubble upwinding} \label{ssec:specialQuad}
  
Using the settings of Section \ref {sec:Quad-Bubbles}, we  choose $\beta$ such that the upper column in the matrix $M^q_{fe}$ of  \eqref{eq:Mq} is zero. This implies
 \[
 \beta= \frac{3}{4} \left (1 -\frac{2\, \varepsilon}{h} \right).
 \]
 To satisfy  $\beta>0$ and \eqref{eq:h-restriction} for a fixed $\varepsilon$, we restrict the  range for $h$  to
 \[
 h > 2.6\, \varepsilon. 
 \]
 This case is interesting because the matrix $M$ of the FE system \eqref{1d-PG-ls}  becomes a bidiagonal lower triangular matrix  
\begin{equation} \label{eq:M4CDut}
M= tridiag  ( -1,  1, 0).
\end{equation}
As a direct consequence of Theorem \ref{thm:PGError-qB} and  $\varepsilon + \frac{2\beta}{3}\, h =h/2$, we have that the solution $u_h$ of the upwinding PG formulation \eqref{eq:1d-modelPG}, satisfies 
\begin{equation}\label{eq:PGerrBq2s}
h^2\, |u-u_h|^2 + 4 |u-u_h|^2_{*,h}  \leq   \frac{57}{4} 
\inf_{p_h\in \M_h} \left ( h^2\, |u-p_h|^2 +  4 |u-p_h|^2_{*,h}\right ).
\end{equation}

In addition, the system $M\, U= F_{pg}$ can be  solved forward to obtain:
\begin{equation}\label{eq:uj}
u_j= (f, \varphi_1 + \varphi_2+ \cdots + \varphi_j) + (f, B_1-B_{j+1}), \ j=1,2,\cdots n-1.
\end{equation}

We introduce the nodal function $\varphi_0$ corresponding to $x_0=0$, i.e., $\varphi_0$ is the continuous piecewise linear function such that $\varphi_0(x_j) =\delta_{0,j} $, $ j=1,2,\cdots n$. Using that $\varphi_o+\varphi_1+\cdots +\varphi_j =1$ on $[0, x_j]$, the formula 
\eqref{eq:uj} leads to 

\begin{equation}\label{eq:eplpicitSol}
u_j= \int_0^{x_j} f(x)\ dx\,  + \int_0^{x_1} f(B_1-\varphi_{0})\ dx\,  + \int_{x_j}^{x_{j+1}} f(\varphi_{j}-B_{j+1})\ dx, \end{equation}
where 
\[
B_1(x)=  3\left (1 -\frac{2\, \varepsilon}{h} \right )\left ( \frac{x}{h}\right ) \left (1- \frac{x}{h}\right ), \ x\in [0, h], \ \text{and} 
\]
\[
 B_{j+1} (x) = B_1(x-jh),\   x \in [x_j,  x_{j+1}], \ j=1,2,\cdots n-1. 
\]
The next result shows that the discrete solution $\displaystyle u_h=\sum_{j=1}^{n-1} u_j\, \varphi_j$ 
is close to the interpolant of $w(x):=  \int_{0}^{x} f(t) \, dt$, hence it is {\it free of non-physical oscillations}. 

\begin{theorem}\label{thm:PGE-qBs}
If $\displaystyle u_h=\sum_{j=1}^{n-1} u_j\, \varphi_j$ is  the solution of the UPG formulation \eqref{eq:1d-modelPG},  with quadratic bubble  test space and  $\beta =\frac{3}{4} \left (1 -\frac{2\, \varepsilon}{h} \right)$, then
\[
 \left | u_j - \int_0^{x_j} f(x)\ dx\right |  \leq \|f\|_{\infty}  \left (2 -\frac{2\, \varepsilon}{h} \right ) h, \ j=1, \cdots, n-1.
\]
\end{theorem}
\begin{proof}
We note  that \[\int_0^{x_1} B_1\,dx  =  \int_{x_j}^{x_{j+1}} B_{j+1}\, dx = \left (1 -\frac{2\, \varepsilon}{h} \right ) \frac{h}{2},  \ \text{and} \  \int_0^{x_1} \varphi_1\ dx  =   \int_{x_j}^{x_{j+1}}  \varphi_j=\frac{h}{2}.
\]
Thus, assuming  $f$ is  continuous on $[0,1]$, by  using the formulas \eqref{eq:eplpicitSol} and the  triangle inequality, we have
\[
\begin{aligned}
 \left | u_j - \int_0^{x_j} f(x)\ dx\right | &= \\
 \left |  \int_0^{x_1} f(B_1-\varphi_{0})\ dx\,  + \int_{x_j}^{x_{j+1}} f(\varphi_{j}-B_{j+1})\ dx\right | & \leq \|f\|_{\infty}  \left (2 -\frac{2\, \varepsilon}{h} \right ) h.
\end{aligned}
\]
\end{proof}
Theorem \ref{thm:PGE-qBs} proves that the components  $u_j$ of the PG solution \eqref{eq:eplpicitSol} approximate $w(x_j)=\int_0^{x_j} f(x)\ dx$ with $\mathcal{O}(h)$. 
 If$f$ is independent of $\varepsilon$, then $w$ is independent of $\varepsilon$, and consequently, the PG solution given by \eqref{eq:eplpicitSol} is {\it free of non-physical oscillations}.



\subsection{Upwinding PG with exponential bubble functions} \label{sec:Exponential-Bubbles} 

As presented in \cite{connections4CD}, we consider the model problem \eqref{VF1d}  with the discrete space \\ 
 $\M_h= span\{ \varphi_j\}_{j = 1}^{n-1}$ and a basis for $V_h$ obtained by modifying the basis of $\M_h$ using  {\it exponential  bubble functions}.  We define the bubble function $B$ on $[0, h]$ as the solution of 
 \begin{equation}\label{eq:expB}
 -\varepsilon B'' -B' =1/h, \ B(0)=B(h)=0.
 \end{equation}
  Using the function $B$ and the general construction of Section \ref{sec:PG}, we define the set of bubble functions $\{B_1, B_2, \cdots,B_n\}$  on $[0, 1]$ by translations of the function $B$. The test space $V_h$ is defined by 
 \begin{equation}\label{eq:VhE}
 V_h:= span \{ \varphi_j + (B_{j}-B_{j+1}) \}_{j= 1}^{n-1}=span \{ g_j \}_{j= 1}^{n-1},
  \end{equation}
where $ g_j:=\varphi_j + (B_{j}-B_{j+1})$, $j=1,2,\cdots,n-1$.
The idea of using a {\it local dual problem} for building the trial space is also presented  in  Section 2.2.3 of \cite{roos-stynes-tobiska-96}, where \cite{hemker77}, an earlier reference, is acknowledged. However, in \cite{hemker77, roos-stynes-tobiska-96} the concept of {\it discrete Green's function} was used  to produce basis functions that span the test space $V_h$. Here, we managed to build a basis for our test  space $V_h$ using the general construction of Section \ref{sec:PG} with the bubble $B$ defined in \eqref{eq:expB}.

In order to deal with efficient computations of coefficients and the finite element matrix of the {\it exponential bubble UPG method}, we introduce the following notation
 \begin{equation}\label{eq:g0}
 g_0:=\tanh\left (\frac{h}{2\varepsilon}\right )= \frac{ e^{\frac{h}{2\varepsilon}} -e^{-\frac{h}{2\varepsilon}}} {e^{\frac{h}{2\varepsilon}} +e^{-\frac{h}{2\varepsilon}}}= \frac  {1- e^{-\frac{h}{\varepsilon}} }{1+e^{-\frac{h}{\varepsilon}} },
  \end{equation}
 \begin{equation}\label{eq:ld}
 l_0:=\frac{1+g_0}{2 g_0} \   \text{and} \ u_0:=\frac{1-g_0}{2 g_0}.
   \end{equation} 
 The unique solution of \eqref{eq:expB} is 
 \begin{equation}\label{expB}
 B(x)=l_0 \left (1 - e^{-\frac{x}{\varepsilon}} \right )- \frac{x}{h}, \ x \in [0, h].
  \end{equation}
  It is easy to check that
   \begin{equation}\label{IntexpB}
   \int_0^{h} B(x)\, dx = \frac{h}{2 g_0} - \varepsilon,  \ \text{and} \  \int_0^{h} (B'(x))^2\, dx =\frac{1}{2 \varepsilon} 
   \left ( \frac{1}{g_0} - \frac{2 \varepsilon}{h} \right ) = \frac{b_2}{h}.
     \end{equation}
 Thus, \eqref{b1}  holds with $b_1= \frac{1}{2 g_0} - \frac{\varepsilon}{h}$  and  \eqref{b2} holds with $b_2= \frac{1}{g_0} \frac{h}{2 \varepsilon}-1$. 
  In this case, we have 
  \[
  d=d_{\varepsilon,h} =\varepsilon + h\, b_1=\frac{h}{2 g_0}, \ \text{and}\ 1+b_2= \frac{h}{2\varepsilon} \frac{1}{g_0}.
  \] 
 According to   \eqref{eq:COptNorm-hB} the optimal norm  on $\M_h$ is given by 
  \begin{equation}\label{eq:COptNorm-hBexp} 
\|u_h\|_{*,h}^2  =2g_0 \frac{\varepsilon}{h}  \left (  \frac{h^2}{4 g_0^2} |u_h|^2 +  |u_h|^2_{*,h} \right ).
\end{equation}

 Since $tanh(x) \in (0, 1) $ for $x>0$, the condtion \eqref{eq:h-restriction} is  satisfied with no restriction for $h$. Consequently, we have the following result.
  
\begin{theorem}\label{thm:PGError-qB}
If $u$ is the solution of \eqref{VF1d}, $u_h$ the solution of the upwinding PG formulation \eqref{eq:1d-modelPG}   with exponential  bubble  test space, then using the discrete norm \eqref{eq:COptNorm-hBexp}, we have   
\begin{equation}\label{eq:PGerrBq}
\|u - u_h\|_{*,h}\leq \sqrt{ \frac{h}{2\varepsilon} \frac{1}{g_0}}\   \inf_{p_h\in \M_h}\|u - p_h\|_{*,h}.
\end{equation}
Equivalently, by rescaling the estimate we have
\begin{equation}\label{eq:PGerrBqexp}
\begin{aligned}
& {h^2}\,\, |u-u_h|^2 + {4 g_0^2}\, |u-u_h|^2_{*,h} \leq  \\  \leq \frac{h}{2\varepsilon} \frac{1}{g_0} 
\inf_{p_h\in \M_h}   & \left ({h^2}\, |u-p_h|^2 + {4 g_0^2}\,  |u-p_h|^2_{*,h}\right ).
\end{aligned}
\end{equation}
\end{theorem}

For implementation purposes, we include the matrix of the finite element  system  \eqref{1d-PG-ls} with exponential  bubble upwinding.  From  \eqref{eq:M4CDfe}, we get 
  \begin{equation} \label{eq:M4CDfeEB}
M_{fe}^e = tridiag\left ( -\frac{1+g_0}{2g_0},\,  \frac{1}{g_0}, -\frac{1-g_0}{2g_0} \right )=
 tridiag \left (  -l_0,  \frac{1}{g_0}, -u_0 \right).
  \end{equation}  
 
As presented in  \cite{connections4CD}, see also \cite{roos-stynes-tobiska-96, Roos-Schopf15}, the exponential bubble UPG method  produces the exact solution at the nodes, provided that the dual vector is computed exactly. In other words, we have that if  $u_h$ is  the solution of the UPG formulation \eqref{eq:1d-modelPG} with exponential  bubble  test space, then $u_h$ equals the interpolant  $I_h(u)$ of the exact solution of \eqref{VF1d}. 
%


%
For $\varepsilon <<h$, we have that $g_0=g_0(\varepsilon,h) \approx 1$. Thus, using $|v|^2_{*,h} \leq \|v\|^2$ on $H^1_0(0, 1)$, and the standard estimates for the interpolant on uniform meshes, 
$\|u-I_h(u)\| = \mathcal{O}(h^2), |u-I_h(u)| = \mathcal{O}(h)$, by taking $p_h= I_h(u)$ in \eqref{eq:PGerrBqexp} we obtain 
\begin{equation}\label{eq:highorder}
|u-u_h| \leq h^{-1}\, \|u - u_h\|_{*,h}\leq C\, h^{3/2},
\end{equation}
where $C=C(u, \varepsilon)$ is independent of $h$. 

However, if $u_{h,c}$ is the computed solution 
for   the exponential bubble UPG method with  $\varepsilon <<h$, numerical experiments  show that the error $|u-u_{h,c}|$ does not have a well established  order. This can be justified based on the  error in computing $e^{-\frac{h}{\varepsilon}}$. We note that if 
$\frac{h}{\varepsilon}$ is too large, then $e^{-\frac{h}{\varepsilon}}$ is computed as $0$. For example, for {\it double precision arithmetic}, we have that $e^{-36.05}$ is smaller than the $\epsilon_{machine} =2^{-52}$. Thus,  $1+e^{-\frac{h}{\varepsilon}}$ is computed as $1$  for $\frac{h}{\varepsilon} \geq 36.05$. We also  have 
 that for $\frac{\varepsilon}{h} \to 0$, 
\[ 
g_0 \to 1, \ \text{and} \  g_j =\varphi_j +B_j-B_{j+1} \to \chi_{|_{[x_{j-1}, x_j]}}. 
\]
Consequently,  for $\frac{\varepsilon}{h} \to 0$,  we obtain
\[
M_{fe}^e \to  tridiag(-1, 1, 0), \ \text{and} \ (f,g_j) \to  \int_{x_{j-1}}^{x_j} f(x) \, dx. 
\]
Hence,  if $\varepsilon <<h$, the computed  matrix $M_{fe}^e$ becomes  $tridiag(-1, 1, 0)$. 
Using a high order quadrature  to estimate the dual vector of the exponential bubble UPG method, we can get a very  accurate approximation of  $(f,g_j) \approx \int_{x_{j-1}}^{x_j} f(x) \, dx$,  especially  if $f$ is, for example, a polynomial function. 
In conclusion, the  computed linear system  is  very close or identical to the system 
\[
 [tridiag(-1, 1, 0)] \, U = \left [ \int_{x_{0}}^{x_1} f(x) \, dx, \cdots,  \int_{x_{n-2}}^{x_{n-1}} f(x) \, dx \right ]^T.
\]
The system  can be solved exactly to obtain
\[
u_j = \int_{0}^{x_j} f(x) \, dx, \ j=1,2,\cdots,n-1. 
\]
This implies that, when $\varepsilon <<h$,  the computed solution $u_{h,c}\in \M_h$ satisfies
\begin{equation}\label{eq:uhc}
u_{h,c}(x_j) \approx u_j=\int_{0}^{x_j} f(x) \, dx, \  j=1,2,\cdots,n-1. 
\end{equation}
Thus,  $u_{h,c}$  is very close to the interpolant of  $w(x)=  \int_{0}^{x} f(t) \, dt$ on  $[0,x_{n-1}]$. 

For example, for $f=1$  and $b=1$, the exact solution of \eqref{eq:1d-model} is 
\[
u(x)= x- \frac{e^{\frac{x}{\varepsilon}} -1}{e^{\frac{1}{\varepsilon}} -1}.
\]
In this case, $(f,g_j) = (1,g_j) =(1,\varphi_j) =h=\int_{x_{j-1}}^{x_j} f(x) \, dx$ is computed exactly. For $\frac{h}{\varepsilon} \geq 36.05$ the  computed matrix $M_{fe}^e$ of the finite element  system  \eqref{1d-PG-ls}  is $  tridiag(-1, 1, 0)$. Thus, $
  u_{h,c}(x_j) =x_j,   \  j=1,2,\cdots,n-1$. 
Direct computation of $|u-u_{h,c}|$ shows that, for $\varepsilon <<h<1$, we have 
\[
|u-u_{h_c}|^2 =\frac{1}{2\varepsilon}\,  \frac{1+e^{-1/\varepsilon}} {1-e^{-1/\varepsilon}} - \frac{2}{h}\, 
\frac{1-e^{-h/\varepsilon}}{1-e^{-1/\varepsilon}} +\frac{1}{h} \approx \frac{1}{2\varepsilon} -\frac{1}{h}
\]
that could be very large if $\varepsilon <<h$. This example shows that even though the exponential bubble UPG method reproduces the exact solution at the nodes, the energy error could be quite large for $\varepsilon <<h$. 

Therefore,  for the exponential bubble UPG method, because of the sensitivity of exponential evaluations for large negative numbers using computer arithmetic, careful attention is needed in computing the stiffness matrix and the dual vector. In addition, special quadrature are required to estimate the energy error close to the boundary layers. 

Certainly,  the energy  error improves significantly if we integrate away from the boundary layer, for example, on  $[0,x_{n-1}]$ for if $\varepsilon <<h$.
We also note that, for $h<2 \varepsilon$, we have  $g_0=g_0(\varepsilon,h) \approx \frac{h}{2\varepsilon}<1$. In this case, the estimate \eqref{eq:PGerrBqexp} leads to a standard energy  estimate 
\[
| u-u_h| \leq  C\, h, 
\]
with $C=C(u)$ independent of $h$ and $\varepsilon$.




\section{Conclusion and remarks towards extending the results to mutidimensional case}\label{sec:conclusion} 
We analyzed and compared mixed variational  formulations for a model convection-diffusion  problem. The key ingredient in our  analysis  is the representation of the optimal norm on the trial spaces  at the continuous and discrete levels. The  ideas presented  for the one dimensional  model problem can be  used for creating new and  efficient discretizations for the multi-dimensional cases of convection dominated problems.  

Below, we list the most useful ideas that our study concluded to have the potential for helping the design of new and more efficient discretization methods for convection dominated problems in the multidimensional case.

\vspace{0.05in}

I) First, we note that for any type of discrete  variational formulation we use to approximate the solution of \eqref{VF1d}, the discrete solution $u_h$ is independent of the  norms we choose on the test and trial spaces. However, for  the  {\it standard linear} test  space method, the SPLS method and the UPG method, the {\it discrete optimal trial norm}  identifies what can be approximated  with the given choice of test and trial spaces. \\ For example, for the {\it standard linear} test space method, only a weighted energy norm $\varepsilon |u| $ can be used to measure the error.  The second part of the discrete optimal norm  is a  semi-norm that is weaker than the $L^2$-norm, see \eqref{eq:CvsD}. Consequently, we cannot expect  an optimal $L^2$-error  approximation for this discretization. 
 
The  weight for the  energy norm improves from $\varepsilon $ to $\varepsilon + \frac{2\beta}{3} h $ for the quadratic UPG, see \eqref{eq:COptNorm-hBq},  and  to $\frac{h}{2 g_0}$ for the exponential UPG,  see \eqref{eq:COptNorm-hBexp}. As shown by our results in \eqref{eq:PGerrBq2}  and \eqref{eq:PGerrBqexp}, this improvement  leads to {\it better optimal norm estimates} for the UPG discretizations.
\vspace{0.08in}

II) The  {\it continuous and discrete optimal trial norms}  and the dependence  $c_0=c_0(\varepsilon, h)$ in the error estimate   \eqref{eq:Approx-Opt}  can  {\it  predict }  approximability of the continuous solution  for  the  {\it given choice of the discrete test and trial spaces}.  For example, for the {\it standard linear}  test space method, the norms \eqref{eq:COptNorm} and \eqref{eq:COptNorm-h-d} are weak when compared to the  standard {\it  unweighted} $H^1$-norm. In addition, $c_0 \approx \frac{h}{\varepsilon}$ could be very large for $\varepsilon <<h$. For the SPLS method $c_0=1$, and the error approximation  improves when compared with the  {\it standard linear} case, as presented in \cite{ CRD-results, CRD-arxiv}.
\vspace{0.08in}

III)  We can {\it  choose the test space} to {\it create upwinding diffusion from the convection part} in the variational formulation as done in the bubble UPG method. We can see how this idea works by comparing \eqref{dVF} and   \eqref{eq:bPG}.
{\it This can be done at the  basis level} by adding locally supported upwinding functions to each nodal function of the trial space. The {\it upwinding process}  leads to  the elimination of the non-physical oscillation in the discrete solutions, and to better approximation. The idea can be extended to the multidimensional case. 
\vspace{0.08in}

IV) For the $P^1-P^2$-SPLS method, we have  that the  continuous and the discrete optimal trial norms agree and  have a simple representation, making  the error analysis more elegant. The UPG method performs better in spite of the fact that the test space for the UPG is a subspace of the test space  $C^0-P^2$ for the SPLS.  The construction of a test space that creates upwinding diffusion from the convection part  and leads to a simple optimal discrete norm,  remains to be investigated even in the one dimensional case.  
\vspace{0.08in}

V) According to recent work in \cite{CRD-results, connections4CD}, 
the UPG method, the Streamline Diffusion (SD) method, and the Upwinding  Finite Difference (UFD) method may lead to  the same matrix of the resulting linear systems. For an UPG versus  SD comparison see \cite{CRD-results},  and  for an UPG versus UFD comparison see \cite{connections4CD}. In spite of that, for  both comparisons it was observed  that  UPG  performs better than SD, and better than UFD. This can be justified by the fact that the UPG is {\it a global variational} method,  SD is {\it a local residual stabilization} method, while the UFD  is a  {\it a particular version} of the UPG method where the RHS dual vector of UPG is approximated by using a low order quadrature, such as the trapezoid rule, see \cite{connections4CD}. 
\vspace{0.08in}

 We conclude by declaring  the bubble UPG method the  most performant discretization for the one dimensional model, and we hope that the ideas of bubble UPG  approach can be successfully extended to the multidimenional case to outperform  the existing methods for convection dominated problems. 





 \bibliographystyle{plain} 
 
\def\cprime{$'$} \def\ocirc#1{\ifmmode\setbox0=\hbox{$#1$}\dimen0=\ht0
  \advance\dimen0 by1pt\rlap{\hbox to\wd0{\hss\raise\dimen0
  \hbox{\hskip.2em$\scriptscriptstyle\circ$}\hss}}#1\else {\accent"17 #1}\fi}
  \def\cprime{$'$} \def\ocirc#1{\ifmmode\setbox0=\hbox{$#1$}\dimen0=\ht0
  \advance\dimen0 by1pt\rlap{\hbox to\wd0{\hss\raise\dimen0
  \hbox{\hskip.2em$\scriptscriptstyle\circ$}\hss}}#1\else {\accent"17 #1}\fi}

 \end{document}